\documentclass[12pt, A4]{article}
\usepackage{amssymb}
\usepackage{amsmath,amsfonts,amsthm,amstext}
\usepackage[english]{babel}
\usepackage[latin1]{inputenc}
\usepackage{makeidx}
\usepackage{amscd}
\usepackage{graphicx}
\usepackage{fancybox}
\usepackage{niceframe}
\usepackage{times}
\usepackage{mathrsfs}
\usepackage{bbm}
\usepackage[all]{xy}
\usepackage{fancyhdr}
\usepackage{sectsty}
\usepackage[Conny]{fncychap}
\usepackage[symbol]{footmisc}
\usepackage[T1]{fontenc}
\DefineFNsymbols{stars}[text]{* {**} {*} {\S} {\I}}
\setfnsymbol{stars}

\newtheorem{thm}{Theorem}[section]
\newtheorem{cor}[thm]{Corollary}
\newtheorem{pro}[thm]{Proposition}
\newtheorem{deff}[thm]{Definition}
\newtheorem{lem}[thm]{Lemma}
\newtheorem{rem}[thm]{Remark}

\newcommand{\nc}{\newcommand}

\nc{\cc}{\D{C}} \nc{\hh}{\D{H}} \nc{\nn}{\D{N}} \nc{\oo}{\D{O}}
\nc{\qq}{\D{Q}}
 \nc{\rr}{\D{R}}
\nc{\zz}{\D{Z}} \nc{\livre}{\ast}

\def\G{{\cal G}}

\def\dbigcup{\mathinner{\bigcup \mkern -13.2mu \rlap{\raise 0.6ex\hbox{.}}\mkern 14.9mu}}

\nc{\barr}{\begin{array}} \nc{\earr}{\end{array}}
\nc{\bthm}{\begin{thm}} \nc{\ethm}{\end{thm}}
\nc{\bpro}{\begin{pro}} \nc{\epro}{\end{pro}}
\nc{\blem}{\begin{lem}} \nc{\elem}{\end{lem}}
\nc{\bins}{\begin{ins}} \nc{\eins}{\end{ins}}
\nc{\bcor}{\begin{cor}} \nc{\ecor}{\end{cor}}
\nc{\brem}{\begin{rem}} \nc{\erem}{\end{rem}}
\nc{\bdeff}{\begin{deff}} \nc{\edeff}{\end{deff}}
\nc{\bea}{\begin{eqnarray}} \nc{\eea}{\end{eqnarray}}
\nc{\D}[1]{{\mathbb#1}}
\def\R{\rm I\kern -.2em R}
\def\N{\rm I\kern -.18em N}
\def\Z{\rm Z\kern -.332em Z}
\def\de{\rm [\kern -.15em [}
\def\dd{\rm ]\kern -.15em ]}
\def\||{\hspace{0.15cm}|\hspace{0.15cm}}

\title{The Fundamental group of a finite graph of conjugacy separable groups with finite edge groups is conjugacy separable}

\author{S. C. Chagas\footnote{\vspace*{-.5cm} The author was supported by
CNPq.}
  }

\begin{document}
\maketitle
\begin{abstract} The main objective of this paper  is to give a positive answer to the natural question proposed by  Ashot Minasyan: Is the fundamental group of finite graph of conjugacy separable groups with  finite edge groups  conjugacy separable?
\end{abstract}


\section{Introduction}
\renewcommand{\thefootnote}{\arabic{footnote}}
 \setcounter{footnote}{1}
 A group $G$ is conjugacy separable if whenever $x$ an $y$ are
  non-conjugate elements of $G$, there is a finite homomorphic image of  $G$ in which the images of $x$ and $y$ are
  non-conjugate. 
  
  The notion of the conjugacy separability owes its importance
  to the fact, first pointed out by Mal'cev \cite{M}, that the conjugacy problem has
  a positive solution in finitely presented conjugacy separable groups.

The notion  was introduced by Blackburn \cite{Black} who showed that
finitely generated torsion-free nilpotent groups are conjugate separable.
This was extended to supersolvable groups by Kargapolov \cite{Kargapolov}, and to
finitely generated nilpotent-by-finite groups by Toh.  Independentely Formanek (1975) and Remeslennikov \cite{Rem} 
proved the conjugacy separability of polycyclic groups (both proofs are based on a number theoretic result of  Chevalley (C. CHEVALLEY, Deux thkorknes d'arithmktique, J. Math. Sot. Japan 3 (1951),
36-44). 

The most significant relatively recent results in the diretion is the proof of the conjugacy separability of the 
fundamental groups of 3-manifolds by Hamilton, Wilton, Zalesskii \cite{HWZ13}  and of  right angled Artin groups by   
Minasyan  \cite{Minasyan2012}. The latter result was used  by Minasyan and  Zalesskii  \cite{MinasyanZalesskii2015}  to  show that all hyperbolic virtually compact special groups (in the sense of D.Wise) are conjugacy separable.

The main result of this paper  is to answer the natural question proposed by  Ashot Minasyan: Is the fundamental group of a finite graph of conjugacy separable groups with finite edges groups conjugacy separable? 

\begin{thm}\label{thm:principal}
The fundamental  group of a finite graph of  conjugacy separable groups  with  finite edges groups is conjugacy separable.
\end{thm}

%
%

%

A group G is defined to be hereditarily conjugacy separable if all finite index subgroups of $G$ are conjugacy separable. This concept
is stronger and in some sense more useful than simply conjugacy
separability, in view of many applications, discovered in
\cite{Minasyan2012} and \cite{MM}. For example in \cite{MM}
Minasyan and Martino proved that for finitely generated normal
subgroup $N$ of torsion-free hereditarily conjugacy separable
hyperbolic group $G$ residual finiteness of $G/N$ is equivalent to
conjugacy separability of $N$.

\begin{cor}
	The fundamental group of a finite graph of groups with hereditary conjugacy separable (hsc) vertex groups and finite edge groups is also hsc. 
\end{cor}

 \section{Preliminaries}

The profinite topology on a group $G$ is the topology where the
collection of all finite index normal subgroups of $G$ serves as a
fundamental system of neighborhoods of the identity element $1\in
G$, turning $G$ into a topological group.  Note that for a
subgroup $H$ of $G$, the profinite topology  of $H$ can be
stronger than the topology induced by the profinite topology of
$G$.

The completion $\widehat{G}$ of $G$ with respect to this topology
is called the profinite completion of $G$ and can be expressed as
an inverse limit
$$\widehat{G} =\lim\limits_{\displaystyle\longleftarrow\atop
N} G/N$$
 of all finite quotients of $G$.
 Thus $\widehat{G}$ is a profinite group. Moreover, there exists a
 natural homomorphism $\iota: G \longrightarrow \widehat{G}$  that
 sends $g \mapsto(gN)$;  $\iota$ is a monomorphism when $G$ is
residually finite. If $S$ is a subset of $\widehat{G}$, we denote
by $\overline{S}$ its closure in $\widehat{G}$. The profinite
topology on $G$ is induced by the topology of $\widehat{G}$.

The next proposition expresses the conjugacy separability property
of $G$ in terms of its profinite topology and we shall use it
freely in the paper.
\begin{pro}
 Let $G$ be a group, then the following conditions are equivalent:
 \begin{itemize}
 \item[(i)] $G$ is conjugacy separable;
  \item[(ii)] for each $x\in G$, the conjugacy class of   $x^{G}$ of
  $x$ is closed in the profinite topology. In particular $G$ is residually finite;
  \item[(iii)] $G$ is residually finite and for each pair of elements $x,
y\in G$ such that $y = x^{\gamma}$,  for some $\gamma\in
\widehat{G}$, there exists $g\in G$ such that $y = x^{g}$.
 \end{itemize}
\end{pro}

We shall recall now basic notions of the profinite verstion of Bass-Serre theory of groups acting on trees (see \cite{R-17} for details).

\medskip
Our graphs are oriented  graphs. A graph $\Gamma$ is a set together with
a distinguished subset of {\it vertices} $V= V(\Gamma)$ and together
with two maps $d_0, d_1: \Gamma\longrightarrow V$, which are the
identity when restricted to $V$. This graph is called {\it
	profinite} if $\Gamma$ is a profinite  space (i.e., a compact,
Hausdorff and totally-disconnected topological space), $V$ is a
closed subset of $\Gamma$, and the mappings $d_i$ are continuous.
If $e\in \Gamma$, we say that $d_0(e)$ and $d_1(e)$ are the origin
and terminal vertex of $e$, respectively.  The complement $E=
E(\Gamma)= \Gamma-V(\Gamma)$ of $V(\Gamma)$ in $\Gamma$ is called
the set (space) of {\it edges} of $\Gamma$. For basic concepts
such as connectedness, or of when a graph is a tree, see
\cite[Chapter I]{DD-89}, or   \cite[Part I]{Serre},   for
abstract graphs. We assume that the reader is familiar with basic notions of Bass-Serre theory of groups acting on trees treated in these books. 

We also assume that the reader knows basic facts about profinite groups, in particular the notion of the profinite topology on a group that can be found in \cite[Chapter 3]{RZ-10}. Following the tradition of combinatorial group theory  a subgroup  $H$ of a group $G$ will be called {\it separable}  if it is closed in the profinite topology of $G$.

For a profinite space $X$  that is the inverse limit of finite discrete spaces $X_j$, $[[\widehat{\mathbb{Z}} X]]$ is the inverse limit  of $\widehat 
{[\mathbb{Z}}X_j]$, where $[\widehat{\mathbb{Z}} X_j]$ is the free $\widehat{\mathbb{Z}}$-module  with basis $X_j$. For a pointed profinite space $(X, *)$
that is the inverse limit of pointed finite discrete spaces $(X_j, *)$, $[[\widehat{\mathbb{Z}} (X, *)]]$ is the inverse limit  of $
[\widehat{\mathbb{Z}} (X_j, *)]$, where $[\widehat{\mathbb{Z}} (X_j, *)]$ is the free $\widehat{\mathbb{Z}}$-module  with basis
$X_j \setminus \{ * \}$ \cite[Chapter~5.2]{RZ-10}.

Given a profinite graph $\Gamma$ define the pointed space $(E^*(\Gamma), *)$ as  $\Gamma / V(\Gamma)$ with the image of $V(\Gamma)$ as a distinguished point $*$.
By definition  a profinite tree  $\Gamma$ is a profinite graph with a short exact sequence
$$
0 \to [[\widehat{\mathbb{Z}}(E^*(\Gamma), *)]] \overset{\delta}{\rightarrow} [[\widehat{\mathbb{Z}} V(\Gamma)]] \overset{\epsilon}{\rightarrow} \widehat{\mathbb{Z}} \to 0
,$$
where $\delta(\bar{e}) = d_1(e) - d_0(e)$ for every $e \in E(\Gamma)$, $\bar{e}$ the image of $e$ in $E^*(\Gamma)$ and $\epsilon(v) = 1$ for
every $v \in V(\Gamma)$.

We refer for further details of the profinite version of the Bass-Serre theory to  \cite{R-17}.
If $v$  and $w$ are vertices of a tree  (respectively, of a profinite tree)
$\Gamma$, we denote by $[v,w]$ the smallest subtree  (respectively, a
profinite subtree) of $\Gamma$ containing $v$ and $w$.

A group $H$ is said to act on a graph $\Gamma$ if it acts on $\Gamma$ as a set and if in
addition $d_i(hm)= hd_i(m)$, for all $h\in H$  and $m\in \Gamma$  ($i=0,1$); if $\Gamma$
is a profinite graph and $H$ a profinite group, we assume that the action is continuous.
The quotient $ \Gamma/H$  inherits a natural graph structure (respectively, profinite graph structure).

\medskip 

When we say that $({\cal G}, \Delta)$ is a finite graph of profinite groups we mean that it contains the data of the underlying finite graph, the profinite edge groups, the  profinite vertex groups and the attaching continuous maps. More precisely,
let $\Delta$ be a connected finite graph. A    graph of profinite groups $({\cal G},\Delta)$ over $\Delta$ consists of a specifying profinite group ${\cal G}(m)$ for each $m\in \Delta$, and continuous monomorphisms $\partial_i: {\cal G}(e)\longrightarrow {\cal G}(d_i(e))$ for each edge $e\in E(\Delta)$. The  fundamental group $$\Pi= \Pi_1({\cal G},\Delta)$$ of the finite graph of profinite groups $({\cal G},\Delta)$ is defined by means of a universal property: $\Pi$ is a profinite group together with the following data  and conditions:
\begin{enumerate}
	\item [(i)] a maximal subtree $T$ of $\Delta$;
	\smallskip
	\item [(ii)]  a collection of continuous homomorphisms
	$$\nu_m: {\cal G}(m)\longrightarrow \Pi\quad (m\in \Delta), $$
	and     a continuous  map
	$E(\Delta) \longrightarrow  \Pi$, denoted $e\mapsto t_e$  ($e\in E(\Delta)$), such that
	$t_e=1$, if $e\in E(T)$, and
	$$(\nu_{d_0 (e)}\partial_0)(x)= t_e(\nu_{d_1 (e)}\partial_1)(x)t_e^{-1},\quad  \forall x\in {\cal G}(e), \ e\in E(\Delta); $$
	\smallskip
	\item [(iii)]  the following universal property is satisfied:
	\medskip
	\noindent whenever one has the following data
	\begin{itemize}
		\item $H$ is a profinite group,\\
		\item $\beta_m: {\cal G}(m)\longrightarrow \Pi\quad (m\in \Delta)$
		a collection of continuous homomorphisms,\\
		\item a map $e\mapsto s_e$ ($e\in E(\Delta)$)  with $s_e=1$, if
		$e\in E(T)$, and\\
		\item $(\beta_{d_0 (e)}\partial_0)(x)= s_e(\beta_{d_1
			(e)}\partial_1)(x)s_e^{-1}, \forall x\in {\cal G}(e), \ e\in
		E(\Delta),  $\\
		\smallskip
		\noindent then there exists a unique continuous homomorphism $\delta : \Pi\longrightarrow  H$ such that $\delta(t_e)= s_e$
		$(e\in E(\Delta))$, and for each $m\in\Delta$ the diagram
	\end{itemize}
	
	\medskip
	
	$$\xymatrix{&
		\Pi  \ar[dd]^\delta   \\  {\cal G}(m)  \ar[ru]^{\nu_m}
		\ar[rd]_{\beta_m }\\ &H }$$
	
	\medskip
	\noindent commutes.
\end{enumerate}

In \cite[paragraph (3.3)]{ZM},  the fundamental group
$\Pi$ is  defined explicitly in terms of generators and relations.  It is also proved there
that the definition given above is independent of the choice of
the maximal subtree $T$.
We use the notation $\Pi(m) = {\rm Im}(\nu_m)$.

Associated with the graph of groups $({\cal G}, \Delta)$ there is
a corresponding  {\it standard profinite graph} (or universal covering graph)
$S=S(\Pi)=\dbigcup
\Pi/\Pi(m)$.  The vertices of
$S$ are those cosets of the form
$g\Pi(v)$, with $v\in V(\Delta)$
and $g\in \Pi$; the incidence maps of $S$ are given by the formulas:

$$d_0 (g\Pi(e))= g\Pi(d_0(e)); \quad  d_1(g\Pi(e))=gt_e\Pi(d_1(e)) \,  (\,e\in E(\Delta)).  $$

In fact $S$  is a profinite tree (cf. \cite[Theorem 3.8]{ZM}.
There is a natural  action of
$\Pi$ on $S$, and clearly $ S/\Pi= \Delta$.

\begin{rem}\label{completion}
	If $\pi_1(\G,\Gamma)$ is the fundamental group of a finite graph of groups then one has the induced graph of profinite completions of edge  and vertex groups $(\widehat\G,\Gamma)$ and  a natural homomorphism $\Pi=\pi_1(\G,\Gamma)\longrightarrow \Pi_1(\widehat\G,\Gamma)$. It is an embedding if $\pi_1(\G,\Gamma)$ is residually finite. In this case $\Pi_1(\widehat \G,\Gamma)=\widehat{\pi_1(\G,\Gamma)}$ is simply the profinite completion. 
	Moreover, 
	
	\begin{enumerate}
		\item[(i)]  The tree $S(\Pi)$
		naturally embeds in $S(\widehat\Pi)$ if and only if  the edge and vertex groups $\G(e)$, $\G(v)$     are separable in  $\pi_1(\G,\Gamma)$, or equivalently $\G(e)$ are closed in $\G(d_0(e))$,   $\G(d_1(v))$    with
		respect to the topology induced by the profinite topology on $\Pi$ (see
		\cite[Proposition 2.5]{CB-13}). In this case, $S(\Pi)/\Pi = S(\widehat{\Pi})/\widehat{\Pi} = \Delta$.
		
		\item[(ii)]  If $H$ is an infinite finitely generated subgroup of $\Pi$ then  by   combination of Theorem 4.12 and Proposition 4.13  of Chapter 1 in \cite{DD-89} there exists a minimal $H$-invariant subtree $T_H$ of $S(\Pi)$ and 
		it is unique. Moreover, $T_H/H$ is finite.
		
		\item[(iii)]  If  $S(\Pi)$
		naturally embeds in $S(\widehat \Pi)$,  the closure $\overline T_H$ in $S(\widehat \Pi)$ is a $\overline H$-invariant profinite subtree and by  \cite[Lemma 1.5]{RZ-10} contains a unique (in $S(\widehat \Pi)$) minimal $\overline H$-invariant subtree  $\widehat T_{\overline H}$. Moreover,  $\widehat T_{\overline H}/\overline H$ is finite since it is a subgraph of a quotient graph $\overline T_H/\overline H$ of  the finite graph $T_H/H$.
	\end{enumerate}
	
\end{rem}

\begin{lem}\label{Pro831Ribes}[Proposition 8.3.1 \cite{R-17}]  Within the hypotheses of Remark \ref{completion}. Let $b\in \Pi^{abs}$ be a hyperbolic  element of $\Pi^{abs}$ and let  $L_b$ be  the corresponding Tits line. Then the following assertions hold.
\begin{itemize}
	\item[(a)] $\langle b^n\rangle \backslash L_b = \overline{\langle b^n\rangle}  \backslash \overline{L_b}$, for all natural numbers $n=1, 2, \cdots$ 
	\item[(b)] $\overline{L_b}$ is the unique minimal $\langle b^n\rangle$-invariant profinite subtree of $S$ and $\overline{L_b}\cap S^{abs}= L_b$.
	\item[(c)] If $\beta \in  \overline{\langle b\rangle}$ and $\beta w\in L_{b}$ for some $w\in L_b$, then $\beta\in \langle b\rangle$.
	\item[(d)] If $\beta \in  \overline{\langle b\rangle} \setminus \langle b\rangle$, then $\beta L_{b}\cap S^{abs} = \emptyset$.
	\item[(e)] Let $\{ \beta_{\lambda}\,|\, \lambda\in \Lambda\}$ be a complete set of  representatives of the cosets of $\langle b\rangle$ in $\overline{\langle b\rangle}$ (a transversal). Then $\overline{L_b}= \underset{\lambda\in\Lambda}{\cup}\beta_{\lambda}L_b$.
In other words, the abstract graphs $\beta_{\lambda}L_b$ are  the distinct connected components of $\overline{L_b}$ considered as an abstract graph; in particular, $L_b$ is its own connected  component in $\overline{L_b}$ as an abstract graph.
\item[(f)] Let $N=\{ x\in \Pi^{abs}\,|\, L_b= L_b\}$. Then $N$ is closed in the profinite topology of $\Pi^{abs}$.
	\end{itemize}
\end{lem}

 \section{Proof of Theorem \ref{thm:principal}}
\begin{thm}\label{Edges finite} 
A fundamental  group of a finite graph of  conjugacy separable groups  with  finite edges groups is conjugacy separable.
\end{thm}

\begin{proof}
Let  $G=\pi_1(\mathcal{G},\Gamma)$ be a fundamental  group of a graph of groups  $(\mathcal{G},\Gamma)$, where the vertex groups $G_v$ are  conjugacy separable groups and  $G_e$ is finite.

Consider the action of $G$ on the Bass-Serre tree $S$ associated with this splitting and the continuous action of $\widehat{G}= \Pi_1(\mathcal{G},\Gamma)$ on the profinite tree $S(\widehat{G})$ (see \ref{completion}).

\medskip
Let $g_1, g_2$ be elements of $G$ such that $g_1^{\gamma}=g_2$ for some $\gamma$ in $\widehat{G}$.

\bigskip
 \noindent {\it Case $ 1 $} (non-hyperbolic).
Suppose   $g_1$ is conjugate to an element of
$G_v$ in $G$, and so we can assume that $g_1\in G_v$. Since $g_2$ is conjugate to the element  $g_1$  in ${\widehat G}$,  by \cite[Lemma 2.8]{R-Z} or   \cite[Lemma 8.3.2]{R-17}  $g_2$ is conjugate in $\widehat{G}$ to a vertex group $\widehat{G}_w\cap G= G_w$ and so we may assume that $g_2\in G_w$.  If $\gamma\in {\widehat G}_v$, then  $g_2\in \widehat{G}_v\cap G= G_v$ and since $G_v$ is conjugacy separable, there exists $g\in G_v$ such that $g_1^g=g_2$ and we are done. 
\medskip     

Otherwise $g_1$ and $g_2$ are conjugate in $\widehat{G}$ to elements of some edge group  by Theorem 7.1.4 in \cite{R-17} and so  $g_1$ and $g_2$ are of finite order.
\medskip

Let $N$ be a finite index normal subgroup of $G$ that intersect all edge groups trivially. Since 
$\widehat{G}= G\widehat{N}$, $\gamma= g\gamma_0$, where $\gamma_0\in \widehat{N}$, $g\in G$; so replacing $g_1$ by $g_1^{g}$ and $\gamma$ by $\gamma_0$ we may assume that $\gamma\in \widehat{N}$.  Since $\gamma\in\widehat{N}$ and $g_1\in \langle g_1\rangle \subseteq \widehat{N} \langle g_1\rangle$ then $g_2\in \widehat{N} \langle g_1\rangle \cap G= N \langle g_1\rangle$. Hence we can reduce our considerations to the case where  $G= N\rtimes \langle g_1\rangle$ and so $G$ is now the fundamental group of a gaph of groups $(\mathcal{G},\Gamma)$ with finite cyclic edge groups.

We show now that $g_1$ and $g_2$ are conjugate to the same vertex group.  If not then there exist maximal disjoint subgraphs $\Gamma_1$ and $\Gamma_2$ of $\Gamma$ such that $g_1$ is conjugate but not $g_2$ into $\pi_1({\cal{G}}, \Gamma_1)$  and $g_2$ is conjugate into $\pi_1({\cal{G}}, \Gamma_2)$ but not $g_1$. Then colapsing $\Gamma_1$ and $\Gamma_2$ and putting $\pi_1({\cal{G}}, \Gamma_1)$  and $\pi_1({\cal{G}}, \Gamma_2)$  on top of obtained new vertices $v_1$ and $v_2$ we get from \cite[ Lemma 2.3 (iii)]{R-Z} that $g_1$  and $g_2$ are not conjugate in $\widehat{G}$, a contradiction. Thus  $g_1$ and $g_2$ are  conjugate into some vertex group $G_v$ and so we may assume that $g_1, g_2\in G_v$.  Since $\widehat{N}$ intersects $G_v$ trivially and $g_1^{\gamma}= g_2$ we deduce that $g_1= g_2$ as required.
\bigskip

 \noindent {\it Case $ 2 $}.
 Suppose $g_1$
does not stabilize any vertex of $S(G)$ and therefore by \cite[Lemma 2.8]{R-Z}  or in \cite[ Lemma 8.3.2]{R-17} $g_2$ does not stabilize any vertex as well.

 By   \cite[Proposition 3.4]{Serre} there exists the infinite straight lines $T_{1}, T_2$ on which $g_1, g_2$  acts freely.  Since  $\overline{T_i}$  are the unique minimal $g_i$-invariant subtree
of $S(\widehat G)$ by Lemma \ref{Pro831Ribes} (b),  $C_{\widehat G} (g_i)$
acts naturally on  $\overline{T_i}$, and  $\gamma^{-1}\overline{T_1}= \overline{T_2}$.
 
 Since by Lemma \ref{Pro831Ribes} (a) $T_2/\langle g_2\rangle = \overline{T_2}/\overline{\langle g_2\rangle} $ is finite, there exist $\delta\in \overline{\langle g_2\rangle}$ such that $\delta\gamma^{-1} e\in T_2$, where $e\in T_1$. Since edges stabilizes  are finite  using Remark \eqref{completion} (i) we deduce that $g= \delta\gamma^{-1}\in G$, and so $\gamma= g^{-1}\delta$. Therefore $g_1^{\gamma}= g_1^{ g^{-1}\delta}= g_1^{g^{-1}}= g_2$ because $\delta$ centralizes $g_2$.  This finishes the proof in this case. 
\end{proof}

\begin{rem}
	As was observed by Ashot Minasyan the hypothesis of finiteness of graph of groups in Theorem \ref{thm:principal} is essential. Indeed, the alternating  group $A_{\infty}$ of permutations (with finite support) of natural numbers $\mathbb{N}$ is simple and so is not residually finite, in particular, is not conjugacy separable. However, it is an inductive limit (union)  of $\varinjlim_{n\in \N}A_n$ and so can be viewed as the fundamental group of   an infinite tree of groups whose edge group are $A_n$ and the incident vertex groups are $A_n$ and $A_{n+1}$.
	\end{rem}

 \end{document}